\documentclass[reqno]{amsart}
\usepackage{amssymb}

\newtheorem{theorem}{Theorem}[section]
\newtheorem{proposition}[theorem]{Proposition}
\newtheorem{lemma}[theorem]{Lemma}

\newtheorem{cor}[theorem]{Corollary}
\theoremstyle{definition}
\newtheorem{definition}[theorem]{Definition}

\allowdisplaybreaks
\numberwithin{equation}{section}
\begin{document}

\title[$p$-Adic hypergeometric functions and certain weight three newforms]
{$p$-Adic hypergeometric functions and certain weight three newforms}


\author[Sulakashna]{Sulakashna (ORCID: 0009-0008-0441-4792)}
\address{Department of Mathematics, Indian Institute of Technology Guwahati, North Guwahati, Guwahati-781039, Assam, INDIA}
\curraddr{}
\email{sulakash@iitg.ac.in}

\author[Rupam Barman]{Rupam Barman (ORCID: 0000-0002-4480-1788)}
\address{Department of Mathematics, Indian Institute of Technology Guwahati, Assam, India, PIN- 781039}
\email{rupam@iitg.ac.in}

\thanks{}


\subjclass[2010]{11F33, 11G25, 11S80, 11T24, 33C20}
\date{28 March 2024, version-1}
\keywords{supercongruences; hypergeometric series; $p$-adic gamma function; elliptic curves.}
\begin{abstract} 
	For an odd prime $p$ and a positive integer $n$, let ${_n}G_n[\cdots]_p$ denote McCarthy's $p$-adic hypergeometric function. In this article, we prove $p$-adic analogue of certain classical hypergeometric identities and using these identities we express the $p$-th Fourier coefficient of certain weight three newforms in terms of special values of ${_3}G_3[\cdots]_p$. Rodriguez-Villegas conjectured certain supercongruences between values of truncated hypergeometric series and the $p$-th Fourier coefficients of these newforms. As a consequence of our main results, we obtain another proof of these supercongruences which were earlier proved by Mortenson and Sun.  
\end{abstract}
\maketitle
\section{Introduction and statement of results}
For a complex number $a$ and a non-negative integer $n$, the rising factorial $(a)_n$ is defined by $(a)_0:=1$ and $(a)_n:=a(a+1)(a+2)\cdots(a+n-1)$ for $n>0$.
Then, for a non-negative integer $r$, and $a_i, b_i\in\mathbb{C}$ with $b_i\notin\{\ldots, -3,-2,-1, 0\}$,
the classical hypergeometric series ${_{r+1}}F_{r}$ is defined by
\begin{equation}\label{hyper}
	{_{r+1}}F_{r}\left[\begin{array}{cccc}
		a_1, & a_2, & \ldots, & a_{r+1} \\
		& b_1, & \ldots, & b_r
	\end{array}| \lambda
	\right]:=\sum_{k=0}^{\infty}\frac{(a_1)_k\cdots (a_{r+1})_k}{(b_1)_k\cdots(b_r)_k}\cdot\frac{\lambda^k}{k!},
\end{equation}
which converges for $|\lambda|<1$. When we truncate the infinite sum \eqref{hyper} at $k=n$, it is known as a truncated hypergeometric series. We use subscript notation
to denote the truncated hypergeometric series
\[
{_{r+1}}F_{r}\left[\begin{array}{cccc}
	a_1, & a_2, & \ldots, & a_{r+1} \\
	& b_1, & \ldots, & b_r
\end{array}| \lambda
\right]_n:=\sum_{k=0}^{n}\frac{(a_1)_k\cdots (a_{r+1})_k}{(b_1)_k\cdots(b_r)_k}\cdot \frac{\lambda^k}{k!}.
\]
\par In \cite{RV}, Rodriguez-Villegas studied the relationship between the number of points over  $\mathbb{F}_p$ on certain Calabi-Yau manifolds and truncated hypergeometric series which corresponds to a particular period of the manifold. In the same article, he examined $18$ supercongruences where he related the truncated hypergeometric series to the Fourier coefficients of modular form of weight three and four. It was Beukers \cite{beuker} who first observed supercongruences of this type in connection with the Ap\'{e}ry numbers used in the proof of the irrationality of $\zeta(3)$. Ahlgren and Ono \cite{AO} proved Beukers' supercongruence conjecture relating Ap\'{e}ry numbers to the coefficients of a certain weight four newform. All the $14$ supercongruences of Rodriguez-Villegas associated with the modular form of weight four are proved, see for example, \cite{Fuselier-McCarthy, kilbourn, long,mccarthy4}. For a nice survey and more conjectural supercongruences, one can also see \cite{mccarthy5}. 
\par Dedekind's eta function $\eta(z)$ is defined by 
\begin{align*}
	\eta(z):=q^{\frac{1}{24}}\prod_{n=1}^{\infty}(1-q^n),
\end{align*}
where $q:=e^{2\pi iz}$ and $\text{Im}(z)>0$.
The integers $a(n)$, $b(n)$, and $c(n)$ are defined by
\begin{align}
	&\sum_{n=1}^{\infty}a(n)q^n:=\eta^6(4z)\in S_3\left(\Gamma_0(16),\left(\frac{-4}{d}\right)\right),\label{modular1}\\
	&\sum_{n=1}^{\infty}b(n)q^n:=\eta^3(6z)\eta^3(2z)\in S_3\left(\Gamma_0(12),\left(\frac{-3}{d}\right)\right),\label{modular2}\\
	&\sum_{n=1}^{\infty}c(n)q^n:=\eta^2(8z)\eta(4z)\eta(2z)\eta^2(z)\in S_3\left(\Gamma_0(8),\left(\frac{-2}{d}\right)\right)\label{modular3}.
\end{align}
These weight three newforms are related to modular $K3$ surfaces. Rodriguez-Villegas \cite{RV} conjectured that for any prime $p>3$ we have
	\begin{align}
		&\label{eq-1}\sum_{n=0}^{p-1}\frac{(2n)!^3}{n!^6}64^{-n}\equiv a(p)\pmod{p^2},\\
		&\label{eq-2}\sum_{n=0}^{p-1}\frac{(3n)!(2n)!}{n!^5}108^{-n}\equiv b(p)\pmod{p^2},\\
		&\label{eq-3}\sum_{n=0}^{p-1}\frac{(4n)!}{n!^4}256^{-n}\equiv c(p)\pmod{p^2},\\
		&\label{eq-4}\sum_{n=0}^{p-1}\frac{(6n)!}{(3n)!n!^3}1728^{-n}\equiv \gamma(p)a(p)\pmod{p^2},
	\end{align}
where $\gamma(p):=-1$ if $p\equiv5\pmod{12}$ and $\gamma(p):=1$ otherwise.
\par Supercongruence \eqref{eq-1} has already been proved by several authors including Ahlgren \cite{ahlgren2}, Ishikawa \cite{Ishikawa}, Mortenson \cite{mortenson}, and Van Hamme \cite{vH}. The supercongruences \eqref{eq-2}-\eqref{eq-4}  were studied by Mortenson in \cite{mortenson}. Using finite field hypergeometric functions, Mortenson proved \eqref{eq-2} for $p\equiv1\pmod 3$, \eqref{eq-3} for $p\equiv1\pmod 4$ and \eqref{eq-4} for $p\equiv1\pmod 6$. When $p\equiv -1\pmod d$, where $d=3,4,6$, Mortenson's approach only allowed him to show the supercongruences up to sign. For example, for $p\equiv -1\pmod 3$, he proved that
\begin{align*}
\left(\sum_{n=0}^{p-1}\frac{(3n)!(2n)!}{n!^5}108^{-n}\right)^2\equiv b(p)^2\pmod{p^2}.
\end{align*}
Sun \cite{Sun} was the first to prove the remaining cases of \eqref{eq-2}-\eqref{eq-4}. He used another approach, namely Schr\"{o}der polynomials and the Zeilberger algorithm to complete the proof of \eqref{eq-2}-\eqref{eq-4}.
\par
In this article, we study the supercongruences \eqref{eq-2}-\eqref{eq-4} via McCarthy's $p$-adic hypergeometric functions involving the $p$-adic Gamma function and extend Mortenson's approach to give a complete proof of \eqref{eq-2}-\eqref{eq-4}. Let $p$ be an odd prime. Let $\mathbb{F}_p$ be the finite field containing $p$ elements. Let $\varphi$ be the quadratic character on $\mathbb{F}_p$. For a positive integer $n$, let $_{n}G_{n}[\cdots]_p$ denote McCarthy's $p$-adic hypergeometric function (see Definition \ref{defin1} in Section \ref{pre}). Firstly, we establish certain transformations and identities for McCarthy's $p$-adic hypergeometric function $_{n}G_{n}[\cdots]_p$. The following transformation for classical hypergeometric series is due to Kummer \cite[p. 4, Eqn (1)]{bailey}.
\begin{align}\label{kummar}
	&{_2}F_1\left[\begin{array}{cc}
		a, &b\\
		&c
	\end{array}|x
	\right]=\frac{\Gamma(c)\Gamma(c-a-b)}{\Gamma(c-a)\Gamma(c-b)}{_2}F_1\left[\begin{array}{cc}
		a, &b \\
		&a+b+1-c
	\end{array}|1-x
	\right]\notag \\
	&\hspace{1cm}+	\frac{\Gamma(c)\Gamma(a+b-c)}{\Gamma(a)\Gamma(b)}(1-x)^{c-a-b}{_2}F_1\left[\begin{array}{cc}
		c-a, &c-b\\
		&1+c-a-b
	\end{array}|1-x
	\right].
\end{align}
The second author with Saikia \cite{BS3} found a $p$-adic analogue of \eqref{kummar} when $a=\frac{1}{4}, b=\frac{3}{4}$, and $c=1$. In the following theorem, we prove a $p$-adic analogue of \eqref{kummar} when $a=\frac{1}{3}$, $b=\frac{2}{3}$, and $c=1$. 
\begin{theorem}\label{MT-1}
 Let $p>3$ be a prime and $t\in\mathbb{F}_p$ such that $t\neq0,1$. We have   
	\begin{align*}
		{_2}G_2\left[\begin{array}{cc}
			\frac{1}{3},& \frac{2}{3} \vspace{.12cm}\\
			0, & 0
		\end{array}|\frac{1}{t} \right]_p = \varphi(-3)\cdot{_2}G_2\left[\begin{array}{cc}
				\frac{1}{3},& \frac{2}{3} \vspace{.12cm}\\
				0, & 0
		\end{array}|\frac{1}{1-t}
		\right]_p.
	\end{align*} 
\end{theorem}
The following is a $p$-adic analogue of \eqref{kummar} when $a=\frac{1}{6}$, $b=\frac{5}{6}$, and $c=1$.
\begin{theorem}\label{MT-2}
	Let $p>3$ be a prime and $t\in\mathbb{F}_p$ such that $t\neq0,1$. We have   
	\begin{align*}
		{_2}G_2\left[\begin{array}{cc}
			\frac{1}{6}, & \frac{5}{6} \vspace{.12cm}\\
			0,& 0
		\end{array}|\frac{1}{t} \right]_p = \varphi(-1)\cdot{_2}G_2\left[\begin{array}{cc}
			\frac{1}{6}, & \frac{5}{6} \vspace{.12cm}\\
			0, & 0
		\end{array}|\frac{1}{1-t}
		\right]_p.
	\end{align*} 
\end{theorem} 
We prove another identity for McCarthy's $p$-adic hypergeometric functions. We recall Bailey's cubic transformation \cite[Eqn (4.06)]{bailey2}:
\begin{align}\label{baily}
&{_3}F_2\left[\begin{array}{ccc}
	a,&2b-a-1, & a+2-2b \vspace{.12cm}\\
	& b, & a-b+\frac{3}{2}
\end{array}|4x
\right]\notag\\
&=
(1-x)^{-a}\cdot	{_3}F_2\left[\begin{array}{ccc}
		\frac{a}{3},& \frac{a+1}{3}, &\frac{a+2}{3} \vspace{.12cm}\\
	&b,	& a-b+\frac{3}{2} 
	\end{array}|\frac{27x^2}{4(1-x)^3} \right].
\end{align}
In the following theorem we prove a $p$-adic analogue of \eqref{baily} when $a=\frac{1}{2}$ and $b=1$. Let $\delta$ denote the function on $\mathbb{F}_p$ defined by
\begin{align*}
\delta(x):=\left\{
\begin{array}{ll}
1, & \hbox{if $x=0$;} \\
0, & \hbox{otherwise.}
\end{array}
\right.
\end{align*}
\begin{theorem}\label{MT-3}
	Let $p>3$ be a prime and $x\in\mathbb{F}_p$ such that $x\neq0,1$. Then we have   
	\begin{align*}
	p^2\cdot {_3}F_2\left(\begin{array}{ccc}
		\varphi, & \varphi, &  \varphi\vspace*{0.1cm}\\
		& \varepsilon, & \varepsilon
	\end{array}|4x
	\right)_p &=	{_3}G_3\left[\begin{array}{ccc}
			\frac{1}{2}, & \frac{1}{2}, &\frac{1}{2} \vspace{.12cm}\\
			0, & 0, & 0
		\end{array}|\frac{1}{4x} \right]_p \\
	&= \varphi(1-x)\cdot{_3}G_3\left[\begin{array}{ccc}
			\frac{1}{2}, &\frac{1}{6}, & \frac{5}{6} \vspace{.12cm}\\
			0, & 0, & 0
		\end{array}|\frac{-4(x-1)^3}{27x^2}
		\right]_p \\
		&\hspace{.5cm}+ \delta(x+2)\cdot\varphi(-1)\cdot p.
	\end{align*} 
\end{theorem}
Next, using the transformations listed in Theorems \ref{MT-1} and \ref{MT-3}, we express the $p$-th Fourier coefficients of the modular forms defined in \eqref{modular1}-\eqref{modular3} in terms of special values of $_{3}G_{3}[\cdots]_p$. 
\begin{theorem}\label{MT-4}
	Let $p$ be an odd prime. 
	 Then we have
	\begin{align*}
		{_3}G_3\left[\begin{array}{ccc}
			\frac{1}{2}, &\frac{1}{4}, & \frac{3}{4} \vspace{.12cm}\\
			0,& 0, & 0
		\end{array}|1
		\right]_p=c(p).
	\end{align*}
\end{theorem}
\begin{theorem}\label{MT-5}
	Let $p>3$ be a prime . Then we have
	\begin{align*}
		{_3}G_3\left[\begin{array}{ccc}
			\frac{1}{2},&\frac{1}{3}, & \frac{2}{3} \vspace{.12cm}\\
			0, & 0, & 0
		\end{array}|1
		\right]_p=b(p).
	\end{align*}
\end{theorem}
\begin{theorem}\label{MT-6}
	Let $p>3$ be a prime. 
	Then we have
	\begin{align*}
		{_3}G_3\left[\begin{array}{ccc}
			\frac{1}{2},&\frac{1}{6},& \frac{5}{6} \vspace{.12cm}\\
			0,& 0,& 0
		\end{array}|1
		\right]_p=\gamma(p)a(p),
	\end{align*}
where $\gamma(p):=-1$ if $p\equiv5\pmod{12}$ and $\gamma(p):=1$ otherwise.
\end{theorem}
As a corollary of Theorems \ref{MT-4}-\ref{MT-6}, we obtain a complete proof of \eqref{eq-2}-\eqref{eq-4}.
\begin{cor}\label{main_thrm}
	Let $p\geq5$ be a prime. We have
	\begin{align}
		&\label{eq-cor-1}\sum_{n=0}^{p-1}\frac{(3n)!(2n)!}{n!^5}108^{-n}\equiv b(p)\pmod{p^2},\\
		&\label{eq-cor-2}\sum_{n=0}^{p-1}\frac{(4n)!}{n!^4}256^{-n}\equiv c(p)\pmod{p^2},\\
		&\label{eq-cor-3}\sum_{n=0}^{p-1}\frac{(6n)!}{(3n)!n!^3}1728^{-n}\equiv \gamma(p)a(p)\pmod{p^2},
	\end{align}
	where $\gamma(p):=-1$ if $p\equiv5\pmod{12}$ and $\gamma(p):=1$ otherwise.
\end{cor}
\section{Preliminaries and some lemmas}\label{pre}
For an odd prime $p$, let $\mathbb{F}_p$ denote the finite field with $p$ elements.
\subsection{Elliptic curve preliminaries}
Let $E$ be an elliptic curve over $\mathbb{F}_p$ given by the Weierstrass form
\begin{align*}
	E:y^2+a_1xy+a_3y=x^3+a_2x^2+a_4x+a_6.
\end{align*}
Using the substitution $y\mapsto\frac{1}{2}(y-a_1x-a_3)$, we have
\begin{align*}
	E:y^2=4x^3+b_2x^2+2b_4x+b_6,
\end{align*}
where $b_2=a_1^2+4a_2$, $b_4=2a_4+a_1a_3$ and $b_6=a_3^2+4a_6$. Employing $y\mapsto2y$ yields
\begin{align}\label{elliptic}
	E:y^2=x^3+\frac{b_2}{4}x^2+\frac{b_4}{2}x+\frac{b_6}{4}.
\end{align}
The trace of
Frobenius endomorphism $a_p(E)$ of $E$ is given by
\begin{align*}
	a_p(E):=p+1-\#E(\mathbb{F}_p),
\end{align*}
where $\#E(\mathbb{F}_p)$ denotes the number of $\mathbb{F}_p$-points on $E$ including the point at infinity. Next, we recall the notion of a quadratic twist. Let $E$ be an elliptic curve given by
\begin{align*}
	E:y^2=x^3+ax^2+bx+c,
\end{align*}
where $a,b,c\in\mathbb{F}_p$. If $D\in\mathbb{F}_p^\times$, then the $D$-quadratic twist of $E$, denoted by $E^D$, is an elliptic curve given by the equation
\begin{align*}
	E^D:y^2=x^3+Dax^2+D^2bx+D^3c.
\end{align*}
It is known that the traces of Frobenius of $E$ and $E^D$ satisfy the following relation:
\begin{align}\label{twist}
	a_p(E)=\left(\frac{D}{p}\right)a_p(E^D).
\end{align}
\subsection{Multiplicative characters and Gauss sums}
Let $\widehat{\mathbb{F}_p^{\times}}$ be the group of all the multiplicative characters on $\mathbb{F}_p^{\times}$. We extend the domain of each $\chi\in \widehat{\mathbb{F}_p^{\times}}$ to $\mathbb{F}_p$ by setting $\chi(0):=0$
including the trivial character $\varepsilon$. 
Let $\mathbb{Z}_p$ and $\mathbb{Q}_p$ denote the ring of $p$-adic integers and the field of $p$-adic numbers, respectively.
Let $\overline{\mathbb{Q}_p}$ be the algebraic closure of $\mathbb{Q}_p$ and $\mathbb{C}_p$ be the completion of $\overline{\mathbb{Q}_p}$.
We know that $\chi\in \widehat{\mathbb{F}_p^{\times}}$ takes values in $\mu_{p-1}$, where $\mu_{p-1}$ is the group of all the $(p-1)$-th roots of unity in $\mathbb{C}^{\times}$. Since $\mathbb{Z}_p^{\times}$ contains all the $(p-1)$-th roots of unity,
we can consider multiplicative characters on $\mathbb{F}_p^\times$
to be maps $\chi: \mathbb{F}_p^{\times} \rightarrow \mathbb{Z}_p^{\times}$.
Let $\omega: \mathbb{F}_p^\times \rightarrow \mathbb{Z}_p^{\times}$ be the Teichm\"{u}ller character.
For $a\in\mathbb{F}_p^\times$, the value $\omega(a)$ is just the $(p-1)$-th root of unity in $\mathbb{Z}_p$ such that $\omega(a)\equiv a \pmod{p}$.
\par Now, we introduce the Gauss sum and recall some of its elementary properties. For further details, see \cite{evans}. Let $\zeta_p$ be a fixed primitive $p$-th root of unity in $\overline{\mathbb{Q}_p}$. 
Then the additive character
$\theta: \mathbb{F}_p \rightarrow \mathbb{Q}_p(\zeta_p)$ is defined by
\begin{align}
	\theta(\alpha):=\zeta_p^{\alpha}.\notag
\end{align}
For $\chi \in \widehat{\mathbb{F}_p^\times}$, the \emph{Gauss sum} is defined by
\begin{align}
	g(\chi):=\sum\limits_{x\in \mathbb{F}_p}\chi(x)\theta(x) .\notag
\end{align}
\begin{lemma}\emph{(\cite[Eqn. (1.12)]{greene}).}\label{lemma2_1}
	For $\chi \in \widehat{\mathbb{F}_p^\times}$, we have
	$$g(\chi)g(\overline{\chi})=p\cdot \chi(-1)-(p-1)\delta(\chi).$$
\end{lemma}
\begin{theorem}\emph{(\cite[Davenport-Hasse Relation]{evans}).}\label{thm2_2}
	Let $m$ be a positive integer and let $p$ be a prime such that $p\equiv 1 \pmod{m}$. For multiplicative characters
	$\chi, \psi \in \widehat{\mathbb{F}_p^\times}$, we have
	\begin{align}
		\prod\limits_{\chi^m=\varepsilon}g(\chi \psi)=-g(\psi^m)\psi(m^{-m})\prod\limits_{\chi^m=\varepsilon}g(\chi).\notag
	\end{align}
\end{theorem}
\subsection{Hypergeometric functions over finite fields}
For multiplicative characters $A$ and $B$ on $\mathbb{F}_p$,
the binomial coefficient ${A \choose B}$ is defined by
\begin{align}\label{jacobi}
	{A \choose B}:=\frac{B(-1)}{p}J(A,\overline{B})=\frac{B(-1)}{p}\sum_{x \in \mathbb{F}_p}A(x)\overline{B}(1-x),
\end{align}
where $J(A, B)$ denotes the Jacobi sum and $\overline{B}$ is the character inverse of $B$. Let $\delta$ denote the function on $\widehat{\mathbb{F}_p^{\times}}$ defined by
\begin{align*}
	\delta(A):=\left\{
	\begin{array}{ll}
		1, & \hbox{if $A=\varepsilon$;} \\
		0, & \hbox{otherwise.}
	\end{array}
	\right.
\end{align*}
We recall the following properties of the binomial coefficients from \cite{greene}:
\begin{align}\label{eq-0.4}
	{A\choose \varepsilon}={A\choose A}=\frac{-1}{p}+\frac{p-1}{p}\delta(A).
\end{align}
In \cite{greene, greene2}, Greene introduced the notion of hypergeometric series over finite fields famously known as \emph{Gaussian hypergeometric series}. He defined hypergeometric functions over finite fields using binomial coefficients as follows.
\begin{definition}(\emph{\cite[Definition 3.10]{greene}}).\label{greene}
	Let $n$ be a positive integer and $x\in \mathbb{F}_{p}$. For multiplicative characters $A_{0},A_{1},\dots,A_{n},B_{1}, \dots,B_{n}$ on $\mathbb{F}_{p}$, the ${_{n+1}}F_n$-hypergeometric function over $\mathbb{F}_{p}$ is defined by
	\begin{align*}
		&{_{n+1}}F_n\left(\begin{array}{cccc}
			A_0, & A_1, &  \ldots, & A_n \\
			& B_1, & \ldots, & B_n
		\end{array}|x
		\right)_p:=\frac{p}{p-1}\sum_{\chi\in\widehat{\mathbb{F}_{p}^{\times}}} {A_0 \chi \choose \chi}{A_{1}\chi \choose B_1\chi}\cdots {A_{n}\chi \choose B_n \chi}\chi(x).
	\end{align*}
\end{definition}
In \cite{FL}, Fuselier et al. gave another definition of hypergeometric function over finite fields. Firstly they defined a period function as follows. 
\begin{definition}(\emph{\cite[p. 28]{FL}})\label{period}
	Let $n$ be a positive integer and $x\in \mathbb{F}_{p}$. For multiplicative characters $A_{0},A_{1},\dots,A_{n},B_{1}, \dots,B_{n}$ on $\mathbb{F}_{p}$, the ${_{n+1}}\mathbb{P}_n$ period function over $\mathbb{F}_{p}$ is defined by
	\begin{align*}
	&{_{n+1}}\mathbb{P}_n\left[\begin{array}{cccc}
		A_0, & A_1, &  \ldots, & A_n \\
		& B_1, & \ldots, & B_n
	\end{array}|x
	\right]:=\delta(x)\prod_{i=1}^{n}J(A_i,\overline{A_i}B_i)\\
		&\hspace*{3cm}+\frac{p^{n+1}}{p-1}\left(\prod_{i=1}^{n}A_iB_i(-1)\right)\sum_{\chi\in\widehat{\mathbb{F}_{p}^{\times}}} {A_0 \chi \choose \chi}{A_{1}\chi \choose B_1\chi}\cdots {A_{n}\chi \choose B_n \chi}\chi(x).
	\end{align*}
We note that the binomial coefficient defined in \cite{FL} is equal to $(-p)$ times the binomial coefficient defined by Greene \cite{greene}. Since we have used Greene's  definition of binomial coefficient, an extra factor of $(-p)^{n+1}$ is appearing in Definition \ref{period}.
\end{definition}
\begin{definition}(\emph{\cite[Eqn (4.9)]{FL}}).\label{fuselier}
		Let $n$ be a positive integer and $x\in \mathbb{F}_{p}$. For multiplicative characters $A_{0},A_{1},\dots,A_{n},B_{1}, \dots,B_{n}$ on $\mathbb{F}_{p}$, the ${_{n+1}}\mathbb{F}_n$-hypergeometric function over $\mathbb{F}_{p}$ is defined by
	\begin{align*}
	&	{_{n+1}}\mathbb{F}_n\left[\begin{array}{cccc}
			A_0, & A_1, &  \ldots, & A_n \\
			& B_1, & \ldots, & B_n
		\end{array}|x
		\right]\\
		&\hspace*{3cm}:=\frac{1}{\prod_{i=1}^{n}J(A_i,\overline{A_i}B_i)} {_{n+1}}\mathbb{P}_n\left[\begin{array}{cccc}
			A_0, & A_1, &  \ldots, & A_n \\
			& B_1, & \ldots, & B_n
		\end{array}|x
		\right].
	\end{align*}
\end{definition}
\subsection{$p$-adic preliminaries}
Firstly, we recall the $p$-adic gamma function. For further details, see \cite{kob}.
For a positive integer $n$,
the $p$-adic gamma function $\Gamma_p(n)$ is defined as
\begin{align}
	\Gamma_p(n):=(-1)^n\prod\limits_{0<j<n,p\nmid j}j\notag
\end{align}
and one extends it to all $x\in\mathbb{Z}_p$ by setting $\Gamma_p(0):=1$ and
\begin{align}
	\Gamma_p(x):=\lim_{x_n\rightarrow x}\Gamma_p(x_n)\notag
\end{align}
for $x\neq0$, where $x_n$ runs through any sequence of positive integers $p$-adically approaching $x$.
This limit exists, is independent of how $x_n$ approaches $x$,
and determines a continuous function on $\mathbb{Z}_p$ with values in $\mathbb{Z}_p^{\times}$.
Let $\pi \in \mathbb{C}_p$ be the fixed root of $x^{p-1} + p=0$ which satisfies
$\pi \equiv \zeta_p-1 \pmod{(\zeta_p-1)^2}$. For $x \in \mathbb{Q}$, we let $\lfloor x\rfloor$ denote the greatest integer less than or equal to $x$ and $\langle x\rangle$ 
denote the fractional part of $x$, i.e., $x-\lfloor x\rfloor$, satisfying $0\leq\langle x\rangle<1$. Then the Gross-Koblitz formula relates Gauss sums and the $p$-adic gamma function as follows.
\begin{theorem}\emph{(\cite[Gross-Koblitz]{gross}).}\label{thm2_3} For $a\in \mathbb{Z}$, we have
	\begin{align}
		g(\overline{\omega}^a)=-\pi^{(p-1)\left\langle \frac{a}{p-1}\right\rangle}\Gamma_p\left(\left\langle \frac{a}{p-1}\right\rangle\right).\notag
	\end{align}
\end{theorem}
McCarthy's $p$-adic hypergeometric function $_{n}G_{n}[\cdots]_p$ is defined as follows.
\begin{definition}(\emph{\cite[Definition 5.1]{mccarthy2}}). \label{defin1}
	Let $p$ be an odd prime and  $t \in \mathbb{F}_p$.
	For positive integers $n$ and $1\leq k\leq n$, let $a_k$, $b_k$ $\in \mathbb{Q}\cap \mathbb{Z}_p$.
	Then the function $_{n}G_{n}[\cdots]_p$ is defined by
	\begin{align}
		&{_n}G_n\left[\begin{array}{cccc}
			a_1, & a_2, & \ldots, & a_n \\
			b_1, & b_2, & \ldots, & b_n
		\end{array}|t
		\right]_p:=\frac{-1}{p-1}\sum_{a=0}^{p-2}(-1)^{an}~~\overline{\omega}^a(t)\notag\\
		&\hspace{2cm} \times
		\prod\limits_{k=1}^n(-p)^{-\lfloor \langle a_k \rangle-\frac{a}{p-1} \rfloor -\lfloor\langle -b_k \rangle +\frac{a}{p-1}\rfloor} \frac{\Gamma_p(\langle a_k-\frac{a}{p-1}\rangle)}{\Gamma_p(\langle a_k \rangle)}
		\frac{\Gamma_p(\langle -b_k+\frac{a}{p-1} \rangle)}{\Gamma_p(\langle -b_k \rangle)}.\notag
	\end{align}
\end{definition}
We recall some lemmas which will be used in the proof of our main results.
\begin{lemma}\emph{(\cite[Lemma 4.1]{mccarthy2}).}\label{lemma-3_1}
	Let $p$ be a prime. For $0\leq a\leq p-2$ and $t\geq 1$ with $p\nmid t$, we have
	\begin{align}
		\omega(t^{-ta})\Gamma_p\left(\left\langle\frac{-ta}{p-1}\right\rangle\right)
		\prod\limits_{h=1}^{t-1}\Gamma_p\left( \frac{h}{t}\right)
		=\prod\limits_{h=0}^{t-1}\Gamma_p\left(\left\langle\frac{1+h}{t}-\frac{a}{p-1}\right\rangle \right).\notag
	\end{align}
\end{lemma}
\begin{lemma}\emph{(\cite[Lemma 3.4]{BSM})}.\label{lemma-3_5}
Let $p$ be an odd prime. For $0<a\leq p-2$, we have
\begin{align}\label{eq-12}
\Gamma_{p}\left(\left\langle1-\frac{a}{p-1}\right\rangle\right)\Gamma_{p}\left(\frac{a}{p-1}\right) = - \overline{\omega}^{a}(-1).
\end{align}
\end{lemma}
The following lemma relates fractional and integral parts of certain rational numbers which will be used to simplify certain products of the $p$-adic gamma function. 
\begin{lemma}\emph{(\cite[Lemma 2.6]{SB}).}\label{lemma-3_3}
	Let $p$ be an odd prime. Let $d\geq2$ be an integer such that $p\nmid d$. Then, for $1\leq a\leq p-2$, we have
	\begin{align*}
	\left\lfloor\frac{-da}{p-1}\right\rfloor = \sum_{h=1}^{d-1} \left\lfloor\frac{h}{d}-\frac{a}{p-1}\right\rfloor -1.
	\end{align*}
\end{lemma}
The next lemma expresses certain product of values of $p$-adic gamma functions in terms of a character sum.
\begin{lemma}\emph{(\cite[Lemma 3.4]{Fuselier-McCarthy}).}\label{lemma-0.1}
	For $p$ an odd prime and $a\in\mathbb{Z}$, with $0<a<p-1$, we have
	\begin{align*}
		\frac{\Gamma_p\left(\left\langle\frac{a}{p-1}\right\rangle\right)\Gamma_p\left(\left\langle\frac{1}{2}-\frac{a}{p-1}\right\rangle\right)}{\Gamma_p\left(\frac{1}{2}\right)}(-p)^{-\left\lfloor\frac{1}{2}-\frac{a}{p-1}\right\rfloor}=-\sum_{t=2}^{p-1}\omega^a(-t)\varphi(t(t-1)).
	\end{align*}
\end{lemma}
Finally, we recall two theorems. Theorem \ref{thrm-1} is a $p$-adic analogue of a classical identity and Theorem \ref{thrm-2} gives a congruence relation between a $p$-adic hypergeometric function $ {_3}G_3[\cdots]_p$ and a truncated hypergeometric series $ _3F_2[\cdots]_{p-1}$.
\begin{theorem}\emph{(\cite[Theorem 2.5]{Fuselier-McCarthy}).}\label{thrm-1}
	Let $p$ be an odd prime and define $s(p):=\Gamma_p\left(\frac{1}{4}\right)\Gamma_p\left(\frac{3}{4}\right)\Gamma_p\left(\frac{1}{2}\right)^2=(-1)^{\left\lfloor\frac{p-1}{4}\right\rfloor+\left\lfloor\frac{p-1}{2}\right\rfloor}$. For $1\neq x\in\mathbb{F}_p^{\times}$,
		\begin{align*}
		{_3}G_3\left[\begin{array}{ccc}
			\frac{1}{2},& \frac{1}{2}, &\frac{1}{2} \vspace{.12cm}\\
			0, & 0, & 0
		\end{array}|\frac{1}{x} \right]_p &= s(p)\cdot\varphi(2(1-x))\cdot{_3}G_3\left[\begin{array}{ccc}
			\frac{1}{2}, &\frac{1}{4}, & \frac{3}{4} \vspace{.12cm}\\
			0,& 0, & 0
		\end{array}|-\frac{(1-x)^2}{4x}
		\right]_p \\
		& \hspace*{0.8cm} + \delta(x+1)\cdot\varphi(-1)\cdot p.
	\end{align*} 
\end{theorem}
\begin{theorem}\emph{(\cite[Theorem 2.5]{mccarthy1}).}\label{thrm-2}
	Let $2\leq d\in\mathbb{Z}$ and let $p$ be an odd prime such that $p\equiv\pm1\pmod d$. Then
	\begin{align*}
			{_3}G_3\left[\begin{array}{ccc}
			\frac{1}{2},& \frac{1}{d}, &\frac{d-1}{d} \vspace{.12cm}\\
			0, & 0, & 0
		\end{array}|1 \right]_p\equiv 	{_3}F_2\left[\begin{array}{ccc}
		\frac{1}{2}, & \frac{1}{d}, &\frac{d-1}{d} \vspace{.12cm}\\
		 & 1, & 1
	\end{array}|1\right]_{p-1}\pmod{p^2}.
	\end{align*}
\end{theorem}
\section{Proof of Theorems \ref{MT-1}, \ref{MT-2} and \ref{MT-3}}
We first prove a proposition which plays an essential role in the proof of Theorem \ref{MT-1}. This proposition gives a relation between the traces of Frobenius of two elliptic curves.
\begin{proposition}\label{prop-1}
	Let $p>3$ be a prime and $E_t:y^2+3xy+ty=x^3$ be a family of elliptic curves where $t\in\mathbb{F}_p$ such that $t\neq0,1$. Then we have
	\begin{align*}
		a_p(E_t)=\left(\frac{-3}{p}\right)a_p(E_{1-t}).
	\end{align*}
\end{proposition}
\begin{proof}
	We have $E_t: y^2+3xy+ty=x^3$ where $t\in\mathbb{F}_p$ such that $t\neq0,1$. Clearly, $P=(0,0)$ is a point of order 3 on $E_t$. Using \cite[Theorem 12.16]{Washington}, we can find an isogeny $\alpha$ from $E_t$ to $E_{t}^\prime$ such that ker$(\alpha)=\{\infty,P,-P\}$ and $E_{t}^\prime$ is given by
\begin{align*}
	E_t^\prime:y^2+3xy+ty=x^3-15tx+(-27t-7t^2).
\end{align*}
We reduce $E_{t}^\prime$ to the form \eqref{elliptic} and obtain 
\begin{align*}
	E_t^\prime:y^2=x^3+\frac{9}{4}x^2-\frac{27t}{2}x-\frac{27}{4}(t^2+4t).
\end{align*}
We know that trace of Frobenius is invariant under isogeny. Therefore,
\begin{align}\label{eq-1.1}
	a_p(E_t)=a_p(E_t^\prime).
\end{align}
Also, $E_{1-t}:y^2+3xy+(1-t)y=x^3$. We write $E_{1-t}$ in the form \eqref{elliptic} as follows.
\begin{align*}
	E_{1-t}:y^2=x^3+\frac{9}{4}x^2+\frac{3(1-t)}{2}x+\frac{(1-t)^2}{4}.
\end{align*}
The $(-3)$-quadratic twist of $E_{1-t}$ is given by
\begin{align}\label{twist-3}
	y^2=x^3-\frac{27}{4}x^2+\frac{27(1-t)}{2}x-\frac{27}{4}(1-t)^2.
\end{align}
Employing $x\mapsto x+3$ in \eqref{twist-3} yields
\begin{align*}
	y^2=x^3+\frac{9}{4}x^2-\frac{27t}{2}x-\frac{27}{4}(t^2+4t),
\end{align*}
which is the elliptic curve $E_t^\prime$. Therefore, $E_t^\prime$ is a $(-3)$-quadratic twist of $E_{1-t}$. Hence, \eqref{twist} yields
\begin{align}\label{eq-1.2}
	a_p(E_{1-t})=\left(\frac{-3}{p}\right)a_p(E_t^\prime).
\end{align}
Combining \eqref{eq-1.1} and \eqref{eq-1.2}, we obtain the desired result.
\end{proof}
\begin{proof}[Proof of Theorem \ref{MT-1}]
	Let $E_{a_1,a_3}:y^2+a_1xy+a_3y=x^3$ be a family of elliptic curves where $a_1,a_3\in\mathbb{F}_p^\times$. We let $T$ denote a generator of $\widehat{\mathbb{F}_p^{\times}}$. From the proof of Theorem 1.1 of \cite{lennon}, we have
	\begin{align*}
		a_p(E_{a_1,a_3})=-\frac{1}{p}-\frac{1}{p(p-1)}\sum_{l=0}^{p-2}g(T^{-l})^3g(T^{3l})T^l\left(\frac{-a_3}{a_1^3}\right).
	\end{align*}
	Taking $T=\omega$ and then using Gross-Koblitz formula we obtain
	\begin{align*}
		a_p(E_{a_1,a_3})&=-\frac{1}{p-1}-\frac{1}{p(p-1)}\sum_{l=1}^{p-2}\overline{\omega}^l\left(\frac{-a_1^3}{a_3}\right)(-p)^{3\left\langle\frac{l}{p-1}\right\rangle+\left\langle-\frac{3l}{p-1}\right\rangle}\\
		&\hspace*{1cm}\times\Gamma_p\left(\frac{l}{p-1}\right)^3\Gamma_p\left(\left\langle\frac{-3l}{p-1}\right\rangle\right)\\
		&=-\frac{1}{p-1}-\frac{1}{p(p-1)}\sum_{l=1}^{p-2}\overline{\omega}^l\left(\frac{-a_1^3}{a_3}\right)(-p)^{-\left\lfloor\frac{-3l}{p-1}\right\rfloor}\\
		&\hspace*{1cm}\times\Gamma_p\left(\frac{l}{p-1}\right)^3\Gamma_p\left(\left\langle\frac{-3l}{p-1}\right\rangle\right).
	\end{align*}
	Using Lemma \ref{lemma-3_1} with $t=3$ and Lemma \ref{lemma-3_3} with $d=3$, we deduce that
	\begin{align*}
		a_p(E_{a_1,a_3})&=-\frac{1}{p-1}-\frac{1}{p(p-1)}\sum_{l=1}^{p-2}\overline{\omega}^l\left(\frac{-a_1^3}{27a_3 }\right)(-p)^{-\left\lfloor\frac{1}{3}-\frac{l}{p-1}\right\rfloor-\left\lfloor\frac{2}{3}-\frac{l}{p-1}\right\rfloor+1}\\
		&\times\frac{\Gamma_p\left(\frac{l}{p-1}\right)^3\Gamma_p\left(\left\langle\frac{1}{3}-\frac{l}{p-1}\right\rangle\right)\Gamma_p\left(\left\langle\frac{2}{3}-\frac{l}{p-1}\right\rangle\right)\Gamma_p\left(\left\langle1-\frac{l}{p-1}\right\rangle\right)}{\Gamma_p\left(\frac{1}{3}\right)
			\Gamma_p\left(\frac{2}{3}\right)}.
	\end{align*}
	Using \eqref{eq-12} and then adding and subtracting the term under the summation for $l=0$, we have
	\begin{align*}
		a_p(E_{a_1,a_3})&=-\frac{1}{p-1}\sum_{l=0}^{p-2}\overline{\omega}^l\left(\frac{a_1^3}{27a_3}\right)(-p)^{-\left\lfloor\frac{1}{3}-\frac{l}{p-1}\right\rfloor-\left\lfloor\frac{2}{3}-\frac{l}{p-1}\right\rfloor}\\
		&\times\frac{\Gamma_p\left(\frac{l}{p-1}\right)^2\Gamma_p\left(\left\langle\frac{1}{3}-\frac{l}{p-1}\right\rangle\right)\Gamma_p\left(\left\langle\frac{2}{3}-\frac{l}{p-1}\right\rangle\right)}{\Gamma_p\left(\frac{1}{3}\right)
			\Gamma_p\left(\frac{2}{3}\right)}\\
		&={_2}G_2\left[\begin{array}{cc}
			\frac{1}{3}, &  \frac{2}{3}\vspace*{0.1cm}\\
			0, & 0
		\end{array}|\frac{a_1^3}{27a_3}
		\right]_p.
	\end{align*}
	Taking $a_1=3, a_3=t$ and $a_1=3, a_3=1-t$, we obtain
	\begin{align*}
		a_p(E_{3,t})&={_2}G_2\left[\begin{array}{cc}
			\frac{1}{3}, &  \frac{2}{3}\vspace*{0.1cm}\\
			0, & 0
		\end{array}|\frac{1}{t}
		\right]_p,\\
		a_p(E_{3,1-t})&={_2}G_2\left[\begin{array}{cc}
			\frac{1}{3}, &  \frac{2}{3}\vspace*{0.1cm}\\
			0, & 0
		\end{array}|\frac{1}{1-t}
		\right]_p.
	\end{align*}
	Now using Proposition \ref{prop-1}, we complete the proof of the theorem.
\end{proof}
\begin{proof}[Proof of Theorem \ref{MT-2}]
	Let $E_t:y^2=x^3-3x^2+4t$ be a family of elliptic curves where $t\in\mathbb{F}_p$ such that $t\neq0,1$. From \cite[Theorem 3.4]{BS1}, we obtain
	\begin{align*}
		a_p(E_t)&=p\cdot\varphi(t)\cdot{_2}G_2\left[\begin{array}{cc}
			\frac{1}{2}, &  \frac{1}{2}\vspace*{0.1cm}\\
			\frac{1}{3}, &  \frac{2}{3}
		\end{array}|t
		\right]_p\\
		&=-\frac{p\cdot\varphi(t)}{p-1}\sum_{l=0}^{p-2}\overline{\omega}^l\left(t\right)(-p)^{-\left\lfloor\frac{1}{3}+\frac{l}{p-1}\right\rfloor-\left\lfloor\frac{2}{3}+\frac{l}{p-1}\right\rfloor-2\left\lfloor\frac{1}{2}-\frac{l}{p-1}\right\rfloor}\\
		&\times\frac{\Gamma_p\left(\left\langle\frac{1}{2}-\frac{l}{p-1}\right\rangle\right)^2\Gamma_p\left(\left\langle\frac{1}{3}+\frac{l}{p-1}\right\rangle\right)\Gamma_p\left(\left\langle\frac{2}{3}+\frac{l}{p-1}\right\rangle\right)}{\Gamma_p\left(\frac{1}{2}\right)^2\Gamma_p\left(\frac{1}{3}\right)
			\Gamma_p\left(\frac{2}{3}\right)}.
	\end{align*}
Letting $l=\frac{p-1}{2}-k$, we have
\begin{align*}
	a_p(E_t)=-\frac{p}{p-1}\sum_{k=1-\frac{p-1}{2}}^{\frac{p-1}{2}}\overline{\omega}^k\left(\frac{1}{t}\right)(-p)^{-\left\lfloor\frac{1}{6}-\frac{k}{p-1}\right\rfloor-\left\lfloor\frac{5}{6}-\frac{k}{p-1}\right\rfloor-2\left\lfloor\frac{k}{p-1}\right\rfloor-1}\\
	\times\frac{\Gamma_p\left(\left\langle\frac{k}{p-1}\right\rangle\right)^2\Gamma_p\left(\left\langle\frac{1}{6}-\frac{k}{p-1}\right\rangle\right)\Gamma_p\left(\left\langle\frac{5}{6}-\frac{k}{p-1}\right\rangle\right)}{\Gamma_p\left(\frac{1}{2}\right)^2\Gamma_p\left(\frac{1}{3}\right)
		\Gamma_p\left(\frac{2}{3}\right)}.
\end{align*}
Using Lemma \ref{lemma-3_1} with $t=3$ and $a=\frac{p-1}{2}$, and the fact that $\Gamma_p\left(\frac{1}{2}\right)^2=-\varphi(-1)$, we deduce that
\begin{align*}
		a_p(E_t)=-\frac{\varphi(-3)}{p-1}\sum_{k=1-\frac{p-1}{2}}^{\frac{p-1}{2}}\overline{\omega}^k\left(\frac{1}{t}\right)(-p)^{-\left\lfloor\frac{1}{6}-\frac{k}{p-1}\right\rfloor-\left\lfloor\frac{5}{6}-\frac{k}{p-1}\right\rfloor-2\left\lfloor\frac{k}{p-1}\right\rfloor}\\
	\times\frac{\Gamma_p\left(\left\langle\frac{k}{p-1}\right\rangle\right)^2\Gamma_p\left(\left\langle\frac{1}{6}-\frac{k}{p-1}\right\rangle\right)\Gamma_p\left(\left\langle\frac{5}{6}-\frac{k}{p-1}\right\rangle\right)}{\Gamma_p\left(\frac{1}{6}\right)
		\Gamma_p\left(\frac{5}{6}\right)}.
\end{align*}
Since the above sum is invariant under $k\mapsto k+(p-1)$ so we have 
\begin{align}
		a_p(E_t)&=-\frac{\varphi(-3)}{p-1}\sum_{k=0}^{p-2}\overline{\omega}^k\left(\frac{1}{t}\right)(-p)^{-\left\lfloor\frac{1}{6}-\frac{k}{p-1}\right\rfloor-\left\lfloor\frac{5}{6}-\frac{k}{p-1}\right\rfloor-2\left\lfloor\frac{k}{p-1}\right\rfloor}\notag\\
	&\times\frac{\Gamma_p\left(\left\langle\frac{k}{p-1}\right\rangle\right)^2\Gamma_p\left(\left\langle\frac{1}{6}-\frac{k}{p-1}\right\rangle\right)\Gamma_p\left(\left\langle\frac{5}{6}-\frac{k}{p-1}\right\rangle\right)}{\Gamma_p\left(\frac{1}{6}\right)
		\Gamma_p\left(\frac{5}{6}\right)}\notag\\
	&=\varphi(-3)\cdot{_2}G_2\left[\begin{array}{cc}
		\frac{1}{6}, &  \frac{5}{6}\vspace*{0.1cm}\\
		0, & 0
	\end{array}|\frac{1}{t}
	\right]_p.\label{eq-1.3}
\end{align}
Replacing $t$ with $1-t$, we obtain
\begin{align}\label{eq-1.4}
	a_p(E_{1-t})=\varphi(-3)\cdot{_2}G_2\left[\begin{array}{cc}
		\frac{1}{6}, &  \frac{5}{6}\vspace*{0.1cm}\\
		0, & 0
	\end{array}|\frac{1}{1-t}
	\right]_p.
\end{align}
Now, taking $x\mapsto x+1$ in $E_t$ and $E_{1-t}$ gives the following curves, respectively,
\begin{align*}
	E_t:y^2=x^3-3x+4t-2,\\
	E_{1-t}:y^2=x^3-3x+2-4t.
\end{align*}
It is easy to verify that $E_{1-t}$ is a $(-1)$-quadratic twist of $E_t$. Therefore, 
\begin{align}\label{eq-1.5}
	a_p(E_t)=\left(\frac{-1}{p}\right)a_p(E_{1-t}).
\end{align}
Combining \eqref{eq-1.3}, \eqref{eq-1.4}, and \eqref{eq-1.5} we complete the proof of the theorem.
\end{proof}
Next, we prove Theorem \ref{MT-3} using a transformation of Fuselier et al. \cite[Theorem 9.14]{FL}. We remark that there is a typo in \cite[Theorem 9.14]{FL}. The factor $\frac{1}{(q-1)^2}$ should be $\frac{1}{q-1}$, and we recall the result below with this correction.
\begin{proof}[Proof of Theorem \ref{MT-3}]
	For $x\neq0,1$, using \cite[Theorem 9.14]{FL} with $A=B=\varphi$ and Lemma \ref{lemma2_1} for $\chi=\varphi$ and $\chi=\varepsilon$, we have
	\begin{align}\label{eq-1.6}
A_x:={_3}\mathbb{F}_2\left[\begin{array}{ccc}
		\varphi, & \varphi, &  \varphi\vspace*{0.1cm}\\
		& \varepsilon, & \varepsilon
	\end{array}|4x
	\right]=\frac{\varphi(1-x)}{p-1}\sum_{\chi\in\widehat{\mathbb{F}_p^\times}}\frac{g(\varphi\chi^3)g(\overline{\chi})^3}{g(\varphi)}\chi\left(\frac{x^2}{4(x-1)^3}\right)\notag\\
	\hspace*{1cm}+\delta(x+2)\varphi(-1)p.
	\end{align}
Using Davenport-Hasse relation with $\psi=\chi^3$ and $m=2$, we have
\begin{align}\label{eq-1.07}
	g(\varphi\chi^3)=\frac{g(\chi^6)g(\varphi)\overline{\chi}^3(4)}{g(\chi^3)}.
\end{align}
Substituting \eqref{eq-1.07} in \eqref{eq-1.6}, we obtain
\begin{align*}
	A_x=\frac{\varphi(1-x)}{p-1}\sum_{\chi\in\widehat{\mathbb{F}_p^\times}}\frac{g(\chi^6)g(\overline{\chi})^3}{g(\chi^3)}\chi\left(\frac{x^2}{4^4(x-1)^3}\right)+\delta(x+2)\varphi(-1)p.
\end{align*}
Taking $\chi=\omega^a$ yields
\begin{align*}
	A_x=\frac{\varphi(1-x)}{p-1}\sum_{a=0}^{p-2}\frac{g(\overline{\omega}^{-6a})g(\overline{\omega}^a)^{3}}{g(\overline{\omega}^{-3a})}\overline{\omega}^a\left(\frac{4^4(x-1)^3}{x^2}\right)+\delta(x+2)\varphi(-1)p.
\end{align*}
Using Gross-Koblitz formula, we have
\begin{align*}
	A_x&=-\frac{\varphi(1-x)}{p-1}\sum_{a=0}^{p-2}\overline{\omega}^a\left(\frac{4^4(x-1)^3}{x^2}\right)(-p)^{3\left\langle\frac{a}{p-1}\right\rangle+\left\langle\frac{-6a}{p-1}\right\rangle-\left\langle\frac{-3a}{p-1}\right\rangle}\\
	&\hspace{1cm}\times \frac{\Gamma_p\left(\frac{a}{p-1}\right)^3\Gamma_p\left(\left\langle\frac{-6a}{p-1}\right\rangle\right)}{\Gamma_p\left(\left\langle\frac{-3a}{p-1}\right\rangle\right)}+\delta(x+2)\varphi(-1)p\\
	&=-\frac{\varphi(1-x)}{p-1}\sum_{a=1}^{p-2}\overline{\omega}^a\left(\frac{4^4(x-1)^3}{x^2}\right)(-p)^{-\left\lfloor\frac{-6a}{p-1}\right\rfloor+\left\lfloor\frac{-3a}{p-1}\right\rfloor}\frac{\Gamma_p\left(\frac{a}{p-1}\right)^3\Gamma_p\left(\left\langle\frac{-6a}{p-1}\right\rangle\right)}{\Gamma_p\left(\left\langle\frac{-3a}{p-1}\right\rangle\right)}\\
	&\hspace*{1cm}-\frac{\varphi(1-x)}{p-1}+\delta(x+2)\varphi(-1)p.
\end{align*}
Employing Lemma \ref{lemma-3_1} with $t=6$ and $t=3$, and Lemma \ref{lemma-3_3} with $d=3$ and $d=6$, we deduce that
\begin{align*}
	A_x&=-\frac{\varphi(1-x)}{p-1}-\frac{\varphi(1-x)}{p-1}\sum_{a=1}^{p-2}\overline{\omega}^a\left(\frac{4(x-1)^3}{27x^2}\right)(-p)^{-\left\lfloor\frac{1}{6}-\frac{a}{p-1}\right\rfloor-\left\lfloor\frac{5}{6}-\frac{a}{p-1}\right\rfloor-\left\lfloor\frac{1}{2}-\frac{a}{p-1}\right\rfloor}\\
	&\hspace*{0.5cm}\times\frac{\Gamma_p\left(\frac{a}{p-1}\right)^3\Gamma_p\left(\left\langle\frac{1}{6}-\frac{a}{p-1}\right\rangle\right)\Gamma_p\left(\left\langle\frac{5}{6}-\frac{a}{p-1}\right\rangle\right)\Gamma_p\left(\left\langle\frac{1}{2}-\frac{a}{p-1}\right\rangle\right)}{\Gamma_p\left(\frac{1}{6}\right)\Gamma_p\left(\frac{5}{6}\right)\Gamma_p\left(\frac{1}{2}\right)}\\
	&\hspace*{1cm}+\delta(x+2)\varphi(-1)p.
\end{align*}
Since $\overline{\omega}^a(-1)=(-1)^a$, we have $\overline{\omega}^a(-1)(-1)^a=1$. Substituting this value and then adding and subtracting the term under the summation for $a=0$, we deduce that
\begin{align}\label{eq-1.7}
	A_x&=-\frac{\varphi(1-x)}{p-1}\sum_{a=0}^{p-2}\overline{\omega}^a\left(\frac{-4(x-1)^3}{27x^2}\right)(-1)^a(-p)^{-\left\lfloor\frac{1}{6}-\frac{a}{p-1}\right\rfloor-\left\lfloor\frac{5}{6}-\frac{a}{p-1}\right\rfloor-\left\lfloor\frac{1}{2}-\frac{a}{p-1}\right\rfloor}\notag\\
	&\hspace*{0.5cm}\times\frac{\Gamma_p\left(\frac{a}{p-1}\right)^3\Gamma_p\left(\left\langle\frac{1}{6}-\frac{a}{p-1}\right\rangle\right)\Gamma_p\left(\left\langle\frac{5}{6}-\frac{a}{p-1}\right\rangle\right)\Gamma_p\left(\left\langle\frac{1}{2}-\frac{a}{p-1}\right\rangle\right)}{\Gamma_p\left(\frac{1}{6}\right)\Gamma_p\left(\frac{5}{6}\right)\Gamma_p\left(\frac{1}{2}\right)}\notag\\
	&\hspace*{1cm}+\delta(x+2)\varphi(-1)p\notag\\
	&=\varphi(1-x)\cdot{_3}G_3\left[\begin{array}{ccc}
		\frac{1}{2}, & \frac{1}{6}, & \frac{5}{6}\vspace*{0.1cm}\\
		0, & 0, & 0
	\end{array}|\frac{-4(x-1)^3}{27x^2}
	\right]_p+\delta(x+2)\varphi(-1)p.
\end{align}
Also, using Definitions \ref{greene}, \ref{period}, and \ref{fuselier}, we obtain
\begin{align}\label{eq-1.9}
	A_x={_3}\mathbb{F}_2\left[\begin{array}{ccc}
		\varphi, & \varphi, &  \varphi\vspace*{0.1cm}\\
		& \varepsilon, & \varepsilon
	\end{array}|4x
	\right]=\frac{p^2}{J(\varphi,\varepsilon)^2}\cdot {_3}F_2\left(\begin{array}{ccc}
		\varphi, & \varphi, &  \varphi\vspace*{0.1cm}\\
		& \varepsilon, & \varepsilon
	\end{array}|4x
	\right)_p.
\end{align}
Using \eqref{jacobi} and \eqref{eq-0.4}, we have $J(\varphi,\varepsilon)^2=1$. Now, \cite[Eqn (3-9)]{BS3} yields
	\begin{align}
	A_x={_3}G_3\left[\begin{array}{ccc}
		\frac{1}{2}, & \frac{1}{2}, & \frac{1}{2}\vspace*{0.1cm}\\
		0, & 0, & 0
	\end{array}|\frac{1}{4x}
	\right]_p.\label{eq-1.8}
\end{align}
Combining \eqref{eq-1.7}, \eqref{eq-1.9} and \eqref{eq-1.8}, we obtain the required identity.
\end{proof}
For $p\equiv 1\pmod 3$, Greene proved the finite field analogue of \eqref{kummar} in \cite[Theorem 4.4(i)]{greene}. We have $\varphi(-3)=-1$ when $p\equiv 2\pmod 3$ and  $\varphi(-1)=-1$ when $p\equiv 3\pmod 4$. Hence, using Theorems \ref{MT-1} and \ref{MT-2} for $t=\frac{1}{2}$, we obtain the following special values of $_2G_2[\cdots]_p$.
\begin{cor}
	Let $p>3$ be a prime.
	\begin{enumerate}
		\item For $p\equiv2\pmod3$, we have
		\begin{align*}
			{_2}G_2\left[\begin{array}{cc}
				\frac{1}{3},& \frac{2}{3} \vspace{.12cm}\\
				0, & 0
			\end{array}|2
			\right]_p=0.
		\end{align*}
		\item  For $p\equiv3\pmod{4}$, we have
		\begin{align*}
			{_2}G_2\left[\begin{array}{cc}
				\frac{1}{6}, & \frac{5}{6} \vspace{.12cm}\\
				0,& 0
			\end{array}|2
			\right]_p=0.
		\end{align*}
	\end{enumerate}
\end{cor}
Employing Theorem \ref{MT-3}, we find certain special values of ${_3}G_{3}[\cdots]_p$ as listed in the following corollary. For $t\in\mathbb{F}_p$, let
\begin{align*}
	\tilde{G}(t,p):={_3}G_3\left[\begin{array}{ccc}
		\frac{1}{2}, &\frac{1}{6}, & \frac{5}{6} \vspace{.12cm}\\
		0, & 0, & 0
	\end{array}|t
	\right]_p.
\end{align*}
\begin{cor}\label{cor2}
	Let $p>3$ be a prime. We have
	\begin{align*}
		&(i)~~	\tilde{G}\left(\frac{1331}{8},p\right)=\left\{
		\begin{array}{ll}
			\varphi(33)(4x^2-p), & \hbox{if $p\equiv1\pmod4,  ~~x^2+y^2=p,~ x$ is odd;} \\
			-p\varphi(33), & \hbox{if $p\equiv3\pmod4$.}
		\end{array}
		\right.\\
		&(ii)~~	\tilde{G}\left(\frac{125}{27},p\right)=\left\{
		\begin{array}{ll}
			\varphi(10)(4x^2-p), & \hbox{if $p\equiv1,3\pmod8, $ and $x^2+2y^2=p$;} \\
			-p\varphi(10), & \hbox{if $p\equiv 5,7\pmod8$.}
		\end{array}
		\right.\\
		&	(iii)~~	\tilde{G}\left(\frac{125}{4},p\right)=\left\{
		\begin{array}{ll}
			\varphi(5)(4x^2-p), & \hbox{if $p\equiv1\pmod3, $ and $x^2+3y^2=p$;} \\
			-p\varphi(5), & \hbox{if $p\equiv 2\pmod3$.}
		\end{array}
		\right.\\
		&(iv)~~	\tilde{G}\left(-\frac{125}{64},p\right)=\left\{
		\begin{array}{ll}
			\varphi(105)(4x^2-p), & \hbox{if $p\equiv1,2,4\pmod7, $ and $x^2+7y^2=p$;} \\
			-p\varphi(105), & \hbox{if $p\equiv 3,5,6\pmod7$.}
		\end{array}
		\right.\\
		&(v)~~	\tilde{G}\left(\frac{614125}{64},p\right)=\left\{
		\begin{array}{ll}
			\varphi(1785)(4x^2-p), & \hbox{if $p\equiv1,2,4\pmod7, $ and $x^2+7y^2=p$;} \\
			-p\varphi(1785), & \hbox{if $p\equiv 3,5,6\pmod7$.}
		\end{array}
		\right.
	\end{align*}
\end{cor}
\begin{proof}
	$(i)$ Taking $x=-\frac{1}{32}$ in Theorem \ref{MT-3} and then employing \cite[Theorem 6 (ii)]{ono}, we obtain the required values.\\
	$(ii)$ Taking $x=-\frac{1}{4}$ in Theorem \ref{MT-3} and then employing \cite[Theorem 6 (iii)]{ono}, we obtain the required values.\\
	$(iii)$ Taking $x=\frac{1}{16}$ in Theorem \ref{MT-3} and then employing \cite[Theorem 6 (v)]{ono}, we obtain the required values.\\	
	$(iv)$ Taking $x=16$ in Theorem \ref{MT-3} and then employing \cite[Theorem 6 (vi)]{ono}, we obtain the required values.\\	
	$(v)$ Taking $x=\frac{1}{256}$ in Theorem \ref{MT-3} and then employing \cite[Theorem 6 (vii)]{ono}, we obtain the required values.\\
\end{proof}
\section{Proof of Theorems \ref{MT-4}, \ref{MT-5}, and \ref{MT-6}} 
Before proving Theorems \ref{MT-4}, \ref{MT-5}, and \ref{MT-6}, we state a theorem of Beukers and Stienstra that describes the Fourier coefficients of the three modular forms given by \eqref{modular1}-\eqref{modular3}.
\begin{theorem}\emph{(\cite[14.2]{S-B})}.\label{thrm-3}
	If we define $\Phi_4(p):=a(p)$, $\Phi_3(p):=b(p)$, and $\Phi_2(p):=c(p)$, then the $p$-th Fourier coefficients of the modular forms are given by
	\begin{align*}
		\Phi_M(p)=\left\{ \begin{array}{ll}
			0 & \hbox{if $\left(\frac{-M}{p}\right)=-1;$}\\
			4a^2-2p & \hbox{if $\left(\frac{-M}{p}\right)=1,~~ p=a^2+Mb^2.$}
		\end{array}\right.
	\end{align*}
\end{theorem}
Now, we are ready to prove the remaining results.
\begin{proof}[Proof of Theorem \ref{MT-4}]
	Using Theorem \ref{thrm-1} with $x=-1$, we obtain
	\begin{align*}
		{_3}G_3\left[\begin{array}{ccc}
			\frac{1}{2},& \frac{1}{2}, &\frac{1}{2} \vspace{.12cm}\\
			0,& 0, & 0
		\end{array}|-1 \right]_p &= s(p)\cdot{_3}G_3\left[\begin{array}{ccc}
			\frac{1}{2}, &\frac{1}{4}, & \frac{3}{4} \vspace{.12cm}\\
			0, & 0, & 0
		\end{array}|1
		\right]_p  + p\cdot\varphi(-1).
	\end{align*}
It is easy to check that $s(p)=\varphi(2)$. Therefore, we have
\begin{align}\label{eqn-4.0}
	{_3}G_3\left[\begin{array}{ccc}
			\frac{1}{2}, &\frac{1}{4}, & \frac{3}{4} \vspace{.12cm}\\
			0, & 0, & 0
		\end{array}|1
		\right]_p  &=	\varphi(2)\cdot {_3}G_3\left[\begin{array}{ccc}
			\frac{1}{2}, & \frac{1}{2},&\frac{1}{2} \vspace{.12cm}\\
			0, & 0, & 0
		\end{array}|-1 \right]_p - p\cdot\varphi(-2)\notag\\
	&=\varphi(2)\cdot p^2\cdot {_3}F_2\left(\begin{array}{ccc}
		\varphi,& \varphi, &\varphi \vspace{.12cm}\\
	 & \varepsilon, & \varepsilon
	\end{array}|-1 \right)_p - p\cdot\varphi(-2).
\end{align}
From \cite[Theorem 6 (iii)]{ono}, we obtain
\begin{align}\label{eqn-4.1}
{_3}F_2\left(\begin{array}{ccc}
\varphi,& \varphi, &\varphi \vspace{.12cm}\\
& \varepsilon, & \varepsilon
\end{array}|-1 \right)_p=\left\{
\begin{array}{ll}
-\frac{\varphi(2)}{p}, & \hbox{if $p\equiv 5,7\pmod8$;} \\
\frac{\varphi(2)(4x^2-p)}{p^2}, & \hbox{if $p\equiv1,3\pmod8 $ and $x^2+2y^2=p$.}
\end{array}
\right.
\end{align}
Combining \eqref{eqn-4.0}, \eqref{eqn-4.1}, and Theorem \ref{thrm-3} with the fact that $\varphi(-2)=1$ if $p\equiv1,3\pmod8$ and $\varphi(-2)=-1$ if $p\equiv5,7\pmod8$, we obtain the required result.
\end{proof}
\begin{proof}[Proof of Theorem \ref{MT-5}]
	Let $p\equiv1\pmod3$. From \cite[Proposition 2.2]{mccarthy1}, we have
	\begin{align*}
		{_3}G_3\left[\begin{array}{ccc}
			\frac{1}{2}, &\frac{1}{3}, & \frac{2}{3} \vspace{.12cm}\\
			0, & 0,& 0
		\end{array}|1
		\right]_p=p^2\cdot {_3}F_2\left(\begin{array}{ccc}
			\varphi, &\chi_3, & \chi_3^2 \vspace{.12cm}\\
			 &\varepsilon,& \varepsilon
		\end{array}|1
		\right)_p,
	\end{align*}
	where $\chi_3$ is a character on $\mathbb{F}_p$ of order $3$. Employing \cite[Proposition 4.2]{mortenson} with $d=3$, we obtain the desired result. Next, we prove the result for $p\equiv 2\pmod3$. Consider
	\begin{align*}
		A&:={_3}G_3\left[\begin{array}{ccc}
			\frac{1}{2}, &\frac{1}{3}, & \frac{2}{3} \vspace{.12cm}\\
			0, & 0,& 0
		\end{array}|1
		\right]_p\\
		&=-\frac{1}{p-1}-\frac{1}{p-1}\sum_{a=1}^{p-2}(-1)^a(-p)^{-\left\lfloor\frac{1}{3}-\frac{a}{p-1}\right\rfloor-\left\lfloor\frac{2}{3}-\frac{a}{p-1}\right\rfloor-\left\lfloor\frac{1}{2}-\frac{a}{p-1}\right\rfloor}\notag\\
		&\hspace*{0.5cm}\times\frac{\Gamma_p\left(\frac{a}{p-1}\right)^3\Gamma_p\left(\left\langle\frac{1}{3}-\frac{a}{p-1}\right\rangle\right)\Gamma_p\left(\left\langle\frac{2}{3}-\frac{a}{p-1}\right\rangle\right)\Gamma_p\left(\left\langle\frac{1}{2}-\frac{a}{p-1}\right\rangle\right)}{\Gamma_p\left(\frac{1}{3}\right)\Gamma_p\left(\frac{2}{3}\right)\Gamma_p\left(\frac{1}{2}\right)}.
	\end{align*}
Using the fact that $\overline{\omega}^a(-1)=(-1)^a$ and Lemma \ref{lemma-0.1}, we obtain
\begin{align*}
	A&=-\frac{1}{p-1}+\frac{1}{p-1}\sum_{t=2}^{p-1}\varphi(t(t-1))\sum_{a=1}^{p-2}\overline{\omega}^a\left(\frac{1}{t}\right)(-p)^{-\left\lfloor\frac{1}{3}-\frac{a}{p-1}\right\rfloor-\left\lfloor\frac{2}{3}-\frac{a}{p-1}\right\rfloor}\\
	&\hspace*{0.5cm}\times\frac{\Gamma_p\left(\frac{a}{p-1}\right)^2\Gamma_p\left(\left\langle\frac{1}{3}-\frac{a}{p-1}\right\rangle\right)\Gamma_p\left(\left\langle\frac{2}{3}-\frac{a}{p-1}\right\rangle\right)}{\Gamma_p\left(\frac{1}{3}\right)\Gamma_p\left(\frac{2}{3}\right)}.
\end{align*}
Adding and subtracting the term under the summation for $a=0$, we have
\begin{align*}
	A&=-\frac{1}{p-1}+\frac{1}{p-1}\sum_{t=2}^{p-1}\varphi(t(t-1))\sum_{a=0}^{p-2}\overline{\omega}^a\left(\frac{1}{t}\right)(-p)^{-\left\lfloor\frac{1}{3}-\frac{a}{p-1}\right\rfloor-\left\lfloor\frac{2}{3}-\frac{a}{p-1}\right\rfloor}\\
	&\hspace*{0.5cm}\times\frac{\Gamma_p\left(\frac{a}{p-1}\right)^2\Gamma_p\left(\left\langle\frac{1}{3}-\frac{a}{p-1}\right\rangle\right)\Gamma_p\left(\left\langle\frac{2}{3}-\frac{a}{p-1}\right\rangle\right)}{\Gamma_p\left(\frac{1}{3}\right)\Gamma_p\left(\frac{2}{3}\right)}-\frac{1}{p-1}\sum_{t=2}^{p-1}\varphi(t(t-1))\\
	&=-\sum_{t=2}^{p-1}\varphi(t(t-1))\cdot{_2}G_2\left[\begin{array}{cc}
		\frac{1}{3}, &  \frac{2}{3}\vspace*{0.1cm}\\
		0, & 0
	\end{array}|\frac{1}{t}
	\right]_p,
\end{align*}
where the last equality is obtained by using the fact that $\sum_{t=2}^{p-1}\varphi(t(t-1))=-1$. Employing Theorem \ref{MT-1} with the fact that $\varphi(-3)=-1$ for $p\equiv2\pmod3$, we obtain
\begin{align*}
	A=\sum_{t=2}^{p-1}\varphi(t(t-1))\cdot{_2}G_2\left[\begin{array}{cc}
		\frac{1}{3}, &  \frac{2}{3}\vspace*{0.1cm}\\
		0, & 0
	\end{array}|\frac{1}{1-t}
	\right]_p.
\end{align*}
Taking $t\mapsto 1-t$, we obtain
\begin{align*}
	A&=\sum_{t=2}^{p-1}\varphi(t(t-1))\cdot{_2}G_2\left[\begin{array}{cc}
		\frac{1}{3}, &  \frac{2}{3}\vspace*{0.1cm}\\
		0, & 0
	\end{array}|\frac{1}{t}
	\right]_p=-A.
\end{align*}
This yields $2A=0$ and hence $A=0$. Using Theorem \ref{thrm-3}, we complete the proof of the theorem.
\end{proof}
\begin{proof}[Proof of Theorem \ref{MT-6}]
Substituting $x=-2$ in Theorem \ref{MT-3}, we obtain
	\begin{align}\label{eqn-4.2}
		p^2\cdot {_3}F_2\left(\begin{array}{ccc}
			\varphi, & \varphi, &  \varphi\vspace*{0.1cm}\\
			& \varepsilon, & \varepsilon
		\end{array}|-8
		\right)_p=\varphi(3)\cdot{_3}G_3\left[\begin{array}{ccc}
			\frac{1}{2}, & \frac{1}{6}, & \frac{5}{6}\vspace*{0.1cm}\\
			0, & 0, & 0
		\end{array}|1
		\right]_p+\varphi(-1)p.
	\end{align}
From \cite[Theorem 6 (i)]{ono}, we obtain
\begin{align}\label{eqn-4.3}
{_3}F_2\left(\begin{array}{ccc}
\varphi, & \varphi, &  \varphi\vspace*{0.1cm}\\
& \varepsilon, & \varepsilon
\end{array}|-8
\right)_p=\left\{
\begin{array}{ll}
-\frac{1}{p}, & \hbox{if $p\equiv3\pmod4$;} \\
\frac{4x^2-p}{p^2}, & \hbox{if $p\equiv1\pmod4,  ~~x^2+y^2=p,~ x$ is odd.}
\end{array}
\right.
\end{align}
Combining \eqref{eqn-4.2}, \eqref{eqn-4.3}, and Theorem \ref{thrm-3} with the fact that $\varphi(3)=-1$ if $p\equiv5,7\pmod{12}$ and $\varphi(3)=1$ if $p\equiv1,11\pmod{12}$, we complete the proof of the theorem.
\end{proof}
We now prove Corollary \ref{main_thrm}.
\begin{proof}[Proof of Corollary  \ref{main_thrm}] We have 
	\begin{align*}
		{_3}F_2\left[\begin{array}{ccc}
			\frac{1}{2},&\frac{1}{3}, & \frac{2}{3} \vspace{.12cm}\\
			& 1, & 1
		\end{array}|1
		\right]_{p-1}=\sum_{n=0}^{p-1}\frac{(3n)!(2n)!}{n!^5}108^{-n}.
		\end{align*}
Now, Theorem \ref{thrm-2} with $d=3$ and Theorem \ref{MT-5} readily yields \eqref{eq-cor-1}.\\
Similarly, we have 
\begin{align*}
	{_3}F_2\left[\begin{array}{ccc}
		\frac{1}{2}, &\frac{1}{4}, & \frac{3}{4} \vspace{.12cm}\\
		& 1, & 1
	\end{array}|1
	\right]_{p-1}=\sum_{n=0}^{p-1}\frac{(4n)!}{n!^4}256^{-n}.
\end{align*}
Combining Theorem \ref{thrm-2} with $d=4$ and Theorem \ref{MT-4}, we obtain \eqref{eq-cor-2}.\\
Finally, we have
\begin{align*}
	{_3}F_2\left[\begin{array}{ccc}
	\frac{1}{2}, &\frac{1}{6}, & \frac{5}{6} \vspace{.12cm}\\
	& 1, & 1
\end{array}|1
\right]_{p-1}=\sum_{n=0}^{p-1}\frac{(6n)!}{(3n)!n!^3}1728^{-n}.
\end{align*} 
Combining Theorem \ref{thrm-2} with $d=6$ and Theorem \ref{MT-6}, we obtain \eqref{eq-cor-3}.
\end{proof}
\section{acknowledgement} We thank John Cremona for his help in the proof of Proposition 3.1. We also thank Zhi-Wei Sun for sharing his work on Rodriguez-Villegas conjectures.


\begin{thebibliography}{99}	
	
\bibitem{ahlgren2}
S. Ahlgren, {\it A  Gaussian  hypergeometric  series and  combinatorial  congruences}, Symbolic computation, number theory, special function, physics and combinatorics. Dev. Math. 4 (2001).	

\bibitem{AO}
S. Ahlgren and K. Ono, {\it A Gaussian hypergeometric series evaluation and Ap\'{e}ry number congruences}, J. Reine. Angew. Math. 518 (2000), 187--212. 
	
\bibitem{bailey2}
W. Bailey, {\it Products of Generalized hypergeometric series}, Proc. London Maths soc. 28 (1928), no. 4, 242-254.
	
\bibitem{bailey}
W. Bailey, {\it Generalized hypergeometric series}, Cambridge University Press, Cambridge, 1935.
		
\bibitem{BS1}
R. Barman and N. Saikia, {\it $p$-Adic gamma function and the trace of Frobenius of elliptic curves},
J. Number Theory 140 (2014), no. 7, 181--195.
	
\bibitem{BS3} R. Barman and N. Saikia, \textit{Certain character sums and hypergeometric series}, Pac. J. Math. 295 (2018), no. 2, 271--289.

\bibitem{BSM} R. Barman, N. Saikia, and D. McCarthy, {\it Summation identities and special values of hypergeometric series in the $p$-adic setting}, J. Number Theory 153 (2015), 63--84.
	
\bibitem{evans}
B. Berndt, R. Evans, and K. Williams, {\it Gauss and Jacobi Sums}, Canadian Mathematical Society Series of Monographs and Advanced Texts, A Wiley-Interscience Publication, John Wiley \& Sons, Inc., New York, (1998).

\bibitem{beuker}
F. Beukers, {\it Another congruence for the Ap\'{e}ry numbers}, J. Number Theory 25 (1987), 201--210.
	
\bibitem{mccarthy5} M. L. Dawsey and D. McCarthy, {\it Hypergeometric Functions over Finite Fields and Modular Forms: A Survey and New Conjectures}, From operator theory to orthogonal polynomials, combinatorics, and number theory-a volume in honor of Lance Littlejohn's 70th birthday, 285 (2021), 41--56.
	
\bibitem{FL}
J. Fuselier, L. Long, R. Ramakrishna, H. Swisher, and F. Tu, {\it Hypergeometric functions over finite fields}, Mem. Am. Math. Soc. 280 (2022), no. 1382.
	
\bibitem{Fuselier-McCarthy} J. Fuselier and D. McCarthy, {\it Hypergeometric type identities in the $p$-adic setting and modular forms}, Proc. Amer. Math. Soc. 144 (2016), 1493--1508
	
\bibitem{greene}
J. Greene, {\it Hypergeometric functions over finite fields}, Trans. Amer. Math. Soc. 301 (1987), no. 1, 77--101.
	
\bibitem{greene2}
J. Greene, {\it Character Sum Analogues for Hypergeometric and Generalized Hypergeometric Functions over Finite Fields}, Ph.D. thesis, Univ. of Minnesota, Minneapolis, 1984.
	
\bibitem{gross}
B. H. Gross and N. Koblitz, {\it Gauss sum and the $p$-adic $\Gamma$-function}, Annals of Mathematics 109 (1979), 569--581.
	
\bibitem{Ishikawa}
T. Ishikawa, {\it On Beukers' congruence}, Kobe J. Math 6 (1989), 49--52.

\bibitem{kilbourn}
T. Kilbourn, {\it An extension of the Ap\'{e}ry number supercongruence}, Acta Arith. 123
(2006), no. 4, 335--348, DOI 10.4064/aa123-4-3.
	
\bibitem{kob} N. Koblitz, {\it $p$-adic analysis: a short course on recent work}, London Math. Soc. Lecture Note Series, 46. Cambridge University Press, Cambridge-New York, (1980).
	
\bibitem{lennon} C. Lennon, {\it Trace formulas for Hecke operators, Gaussian hypergeometric functions, and the modularity of a threefold}, J. Number Theory 131 (2011), no. 12, 2320--2351.
	
\bibitem{long} L. Long, F. Tu, N. Yui, and W. Zudilin, {\it  Supercongruences for rigid hypergeometric Calabi–Yau threefolds}, Adv. in Math. 393 Article 108058 (2021).
	
\bibitem{mccarthy1}
D. McCarthy, {\it Extending Gaussian hypergeometric series to the $p$-adic setting}, Int. J. Number Theory 8 (2012), no. 7, 1581--1612.
	
\bibitem{mccarthy2}
D. McCarthy, {\it The trace of Frobenius of elliptic curves and the $p$-adic gamma function}, Pac. J. Math. 261 (2013), no. 1, 219--236.
		
\bibitem{mccarthy4} D. McCarthy, {\it On a supercongruence conjecture of Rodriguez-Villegas}, Proc. Amer. Math. Soc. 140 (2012), 2241--2254.
		
\bibitem{mortenson}
E. Mortenson, {\it Supercongruences for truncated ${_{n+1}}F_n$-hypergeometric series with applications to certain weight three newforms}, Proc. Amer. Math. Soc. 133 (2005), no. 2, 321--330.
	
	
\bibitem{ono} K. Ono, \textit{Values of Gaussian hypergeometric series}, Trans. Amer. Math. Soc. 350 (1998), no. 3, 1205--1223.
	
\bibitem{S-B}
J. Stienstra and F. Beukers, {\it On the Picards Fuchs equation and the formal Brauer group of certain elliptic $K3$-surfaces}, Math. Ann. 271 (1985) 269--304.
	
\bibitem{SB}
Sulakashna and R. Barman, {\it Number of $\mathbb{F}_q$-points on Diagonal hypersurfaces and hypergeometric function}, Int. J. Number Theory, accepted for publication. https://arxiv.org/abs/2210.11732. 

\bibitem{Sun}
Z.-W. Sun, {\it On sums involving products of three binomial coefficients}, Acta Arith. 156.2 (2012), pp. 123--141.
	
\bibitem{vH}
L. van Hamme, {\it Proof of a conjecture of Beukers on  Ap\'{e}ry numbers}, Proceedings of the conference on $p$-adic analysis (1987), 189--195.	
		
\bibitem{RV} 
F. R. Villegas, {\it Hypergeometric families of Calabi-Yau manifolds. Calabi-Yau Varieties and Mirror Symmetry(Toronto, Ontario, 2001)}, Fields Inst. Commun. 38 (2001), no. 38, 223--231.
	
\bibitem{Washington}
L. C. Washington, {\it Elliptic Curves: Number Theory and Cryptography}, 2nd Edition, CRC Press, Boca Raton (2008).
\end{thebibliography}
\end{document}